\documentclass[12pt]{article} 
\usepackage[a4paper,margin=1in]{geometry}

\usepackage{mathtools,amssymb,amsthm}
\mathtoolsset{showonlyrefs=true}

\newtheorem{theorem}{Theorem}[section]
\newtheorem{lemma}[theorem]{Lemma}
\newtheorem{corollary}[theorem]{Corollary}

\newtheorem{fact}[theorem]{Fact}

\theoremstyle{definition}
\newtheorem{definition}[theorem]{Definition}
\newtheorem{example}[theorem]{Example}

\theoremstyle{remark}
\newtheorem{remark}[theorem]{Remark}

\usepackage{enumitem}
\setlist[enumerate,1]{label=(\arabic*)}
\setlist[enumerate,2]{label=(\roman*)}

\numberwithin{equation}{section}
\numberwithin{figure}{section}
\numberwithin{table}{section}

\usepackage{multirow}

\usepackage{hyperref}

\usepackage[style=alphabetic,isbn=false,url=false,doi=false,giveninits=true,maxnames=9,sorting=nyt,backend=bibtex]{biblatex} 
\renewbibmacro{in:}{}
\bibliography{reference.bib}

\title{Inequality on the optimal constant of Young's convolution inequality for locally compact groups and their closed subgroups}
\author{Takashi Satomi}
\date{\today} 

\begin{document}

\maketitle

\begin{abstract}

  We define the optimal constant $Y ( p_1 , p_2 ; G )$ of Young's convolution inequality as
  \begin{align}
    Y ( p_1 , p_2 ; G ) := \sup \{ \| \phi_1 * ( \phi_2 \Delta^{1 / p_1'} ) \|_p \mid \phi_1 , \phi_2 \colon G \to \mathbb{C} , \; \| \phi_1 \|_{p_1} = \| \phi_2 \|_{p_2} = 1 \}
  \end{align}
  for a locally compact group $G$ and $1 \leq p_1 , p_2 , p \leq \infty$ with $1 / p_1 + 1 / p_2 = 1 + 1 / p$.
  Here $p'$ is the H\"{o}lder conjugate of $p$, $\| \cdot \|_{ p }$ is the $L^p$-norm on a left Haar measure, and $\Delta \colon G \to \mathbb{R}_{> 0}$ is the modular function.
  The main result of this paper is that $Y ( p_1 , p_2 ; G ) \leq Y ( p_1 , p_2 ; H )$ for any closed subgroup $H \subset G$.
  It follows from this inequality that $Y ( p_1 , p_2 ; G ) \leq Y ( p_1 , p_2 ; \mathbb{R} )^{ \dim G - r ( G ) }$ for any connected Lie group $G$ such that the center of the semisimple part is a finite group such as connected linear Lie groups and connected solvable Lie groups, where $r ( G )$ is the dimension of the maximal compact subgroups of $G$.

\end{abstract}



\section{Introduction}

We write $Y ( p_1 , p_2 ; G )$ for the optimal constant (the optimal ratio of both sides) of Young's convolution inequality on a locally compact group $G$.
The main result of this paper is that $Y ( p_1 , p_2 ; G ) \leq Y ( p_1 , p_2 ; H )$ holds for any closed subgroup $H \subset G$ (Theorem \ref{thm:Young-order-reversing}).
As a corollary, we generalize the estimate of $Y ( p_1 , p_2 ; G )$ from above by Beckner \cite{MR385456}, Fournier \cite{MR461034}, Klein--Russo \cite{MR499945}, and Nielsen \cite{MR1304346} to any connected Lie group $G$ such that the center of the semisimple part is a finite group such as connected linear Lie groups and connected solvable Lie groups.
That is, we have $Y ( p_1 , p_2 ; G ) \leq Y ( p_1 , p_2 ; \mathbb{R} )^{ \dim G - r ( G ) }$, where $r ( G )$ is the dimension of the maximal compact subgroups of $G$ (Corollary \ref{cor:Young-non-compact-dimension}).

We set the H\"{o}lder conjugate $1 \leq p' \leq \infty$ of $1 \leq p \leq \infty$ as
\begin{align}
  \frac{1}{p} + \frac{1}{p'} = 1.
\end{align}
The $L^p$-norm with respect to the left Haar measure $dg$ on the locally compact group $G$ is written as $\| \cdot \|_p$.
We define the convolution $\phi_1 * \phi_2$ of measurable functions $\phi_1 , \phi_2 \colon G \to \mathbb{C}$ as
\begin{align}
  \phi_1 * \phi_2 ( g' )
  := \int_{G}^{} \phi_1 ( g ) \phi_2 ( g^{-1} g' ) dg.
\end{align}
In addition, there exists a unique continuous homomorphism $\Delta \colon G \to \mathbb{R}_{> 0}$ such that
\begin{align}
  \int_{G}^{} \phi ( g^{-1} ) dg
  = \int_{G}^{} \frac{\phi ( g )}{\Delta ( g )} dg. \label{eq:modular-function}
\end{align}
This $\Delta$ is called the modular function of $G$.
Then the notion of the optimal constant $Y ( p_1 , p_2 ; G )$ is defined as follows.

\begin{definition}[{\cite[Section 2]{MR499945} \cite{MR1304346}}]
  \label{def:optimal-constant}

  For any locally compact group $G$ and any $1 \leq p_1 , p_2 \leq \infty$ with
  \begin{align}
    \frac{1}{p_1} + \frac{1}{p_2} \geq 1, \label{eq:optimal-constant-p1-p2-condition}
  \end{align}
  the optimal constant $Y ( p_1 , p_2 ; G )$ is defined as
  \begin{align}
    Y ( p_1 , p_2 ; G ) := \sup \{ \| \phi_1 * ( \phi_2 \Delta^{1/p_1'} ) \|_p \mid \phi_1 , \phi_2 \colon G \to \mathbb{C} , \; \| \phi_1 \|_{p_1} = \| \phi_2 \|_{p_2} = 1 \}.
  \end{align}
  Here $\Delta$ is the modular function of $G$, and $p$ is given by
  \begin{align}
    \frac{1}{p_1} + \frac{1}{p_2}
    = 1 + \frac{1}{p}. \label{eq:optimal-constant-p-definition}
  \end{align}

\end{definition}

The main result of this paper is the following theorem.

\begin{theorem}
  \label{thm:Young-order-reversing}

  Suppose that $1 \leq p_1 , p_2 \leq \infty$ satisfy \eqref{eq:optimal-constant-p1-p2-condition}.
  Then $Y ( p_1 , p_2 ; G ) \leq Y ( p_1 , p_2 ; H )$ holds for any closed subgroup $H \subset G$ of any locally compact group $G$.

\end{theorem}

When $H$ is a normal subgroup of $G$, Theorem \ref{thm:Young-order-reversing} is essentially known.
That is, Cowling--Martini--M\"{u}ller--Parcet proved
\begin{align}
  Y ( p_1 , p_2 ; G )
  \leq Y ( p_1 , p_2 ; H ) Y ( p_1 , p_2 ; G / H ) \label{eq:Young-product-normal}
\end{align}
\cite[Proposition 2.2]{MR4000236}.
Since $Y ( p_1 , p_2 ; G / H ) \leq 1$ by the classical Young's inequality (Example \ref{ex:Young-order-reversing-example} \ref{item:Young-order-reversing-example-classic}), Theorem \ref{thm:Young-order-reversing} follows from \eqref{eq:Young-product-normal}.
The inequality \eqref{eq:Young-product-normal} was essentially proved by Beckner when $G$ is abelian \cite[Section IV.5]{MR385456} (see also Fact \ref{fact:Beckner}), and by Klein--Russo when $G$ is a semidirect product of $H$ and $G / H$ \cite[Lemma 2.4]{MR499945}.

Theorem \ref{thm:Young-order-reversing} has some interesting examples.
For instance, we have the classical Young's inequality by applying Theorem \ref{thm:Young-order-reversing} to the case where $H$ is trivial (Example \ref{ex:Young-order-reversing-example} \ref{item:Young-order-reversing-example-classic}).
In addition, if the identity component $G_0 \subset G$ is open (e.g. $G$ is a Lie group), then we have $Y ( p_1 , p_2 ; G ) = Y ( p_1 , p_2 ; G_0 )$ (Example \ref{ex:Young-order-reversing-example} \ref{item:Young-order-reversing-example-component}) by Theorem \ref{thm:Young-order-reversing}.
Thus, it suffices to consider the identity component to determine the value of $Y ( p_1 , p_2 ; G )$ for any Lie group $G$.
In addition, we have $Y ( p_1 , p_2 ; G ) \leq Y ( p_1 , p_2 ; \mathbb{R} )$ for any locally compact group $G$ which has no open compact subgroup by Theorem \ref{thm:Young-order-reversing}.
This claim is a generalization of the previous results of Fournier, Nielsen, and the author (Corollary \ref{cor:Young-R-compare}).

As a corollary of Theorem \ref{thm:Young-order-reversing}, we bound $Y ( p_1 , p_2 ; G )$ for any connected Lie group $G$ such that the center of the semisimple part is a finite group such as connected linear Lie groups and connected solvable Lie groups by using the dimension $r ( G )$ of the maximal compact subgroups.
That is, the following corollary holds, where $\# G$ is the cardinality of a group $G$, and $Z ( G )$ is the center of $G$.

\begin{corollary}
  \label{cor:Young-non-compact-dimension}

  Let $R \lhd G$ be the radical (the largest connected solvable closed normal subgroup) of a connected Lie group $G$, and $r ( G )$ be the dimension of the maximal compact subgroups of $G$.
  If $\# Z ( G / R ) < \infty$, then
  \begin{align}
    Y ( p_1 , p_2 ; G ) \leq Y ( p_1 , p_2 ; \mathbb{R} )^{\dim G - r ( G )} \label{eq:Young-non-compact-dimension-state}
  \end{align}
  holds for any $1 \leq p_1 , p_2 \leq \infty$ with \eqref{eq:optimal-constant-p1-p2-condition}.

\end{corollary}

Corollary \ref{cor:Young-non-compact-dimension} bounds $Y ( p_1 , p_2 ; G)$ explicitly from above because $Y ( p_1 , p_2 ; \mathbb{R} )$ is determined explicitly by Beckner (Fact \ref{fact:Beckner}).
Although Corollary \ref{cor:Young-non-compact-dimension} is known for some $G$ (Table \ref{tab:Young-non-compact-dimension-compare}), to the best of our knowledge, it contains some new examples such as
\begin{align}
  Y ( p_1 , p_2 ; SL_2 (\mathbb{R}) )
  \leq Y ( p_1 , p_2 ; \mathbb{R} )^2. \label{eq:SL2-bound}
\end{align}
It is known that the equality of \eqref{eq:Young-non-compact-dimension-state} holds for some connected Lie groups such as connected compact Lie groups (Corollary \ref{cor:Young-R-compare}) and connected nilpotent Lie groups (Fact \ref{fact:Nielsen}).
We prove Corollary \ref{cor:Young-non-compact-dimension} in Section \ref{sec:non-compact-dimension} by using an argument similar to the generalization of the Brunn--Minkowski inequality to any Lie group by Jing--Tran--Zhang \cite[Theorem 1.1]{MR4616694}.

Here is the organization of this paper.
In Section \ref{sec:Young-known-result}, we compare some known results with Theorem \ref{thm:Young-order-reversing} and Corollary \ref{cor:Young-non-compact-dimension}.
In Section \ref{sec:boundary}, we show Theorem \ref{thm:Young-order-reversing} when $( p_1 , p_2 )$ is on the boundary of the range that satisfies the assumption.
In Section \ref{sec:Young-order-proof}, we show Theorem \ref{thm:Young-order-reversing} when $( p_1 , p_2 )$ is not on the boundary.
In Section \ref{sec:non-compact-dimension}, we show Corollary \ref{cor:Young-non-compact-dimension} by using Theorem \ref{thm:Young-order-reversing} and the argument of Jing--Tran--Zhang.

\section{Comparison of some known results on the optimal constant}
\label{sec:Young-known-result}

In this section, we compare some known results on the optimal constant $Y ( p_1 , p_2 ; G )$ with Theorem \ref{thm:Young-order-reversing} (Subsection \ref{subsec:Young-known-result-order}) and Corollary \ref{cor:Young-non-compact-dimension} (Subsection \ref{subsec:Young-known-result-non-compact-dimension}).

\subsection{Comparison with Theorem \ref{thm:Young-order-reversing}}
\label{subsec:Young-known-result-order}

In this subsection, we see some relations between known results and Theorem \ref{thm:Young-order-reversing}.
Here are some examples of Theorem \ref{thm:Young-order-reversing}.

\begin{example}
  \label{ex:Young-order-reversing-example}

  \begin{enumerate}
    \item \label{item:Young-order-reversing-example-classic}
          Since the trivial group $\{ e \} \subset G$ containing only the identity element $e \in G$ is a closed subgroup, we have $Y ( p_1 , p_2 ; G ) \leq Y ( p_1 , p_2 ; \{ e \} )$ by Theorem \ref{thm:Young-order-reversing}.
          The equality $Y ( p_1 , p_2 ; \{ e \} ) = 1$ holds by definition and hence we have the classical Young's inequality
          \begin{align}
            Y ( p_1 , p_2 ; G ) \leq 1. \label{eq:Young-order-reversing-example-classic-state}
          \end{align}
          There are at least two proofs of \eqref{eq:Young-order-reversing-example-classic-state}.

          \begin{enumerate}

            \item \label{item:Young-order-reversing-example-classic-Riesz-Thorin}
                  The method to deduce the case of $p_1 = 1 , p_2'$ by using the Riesz--Thorin theorem.

            \item \label{item:Young-order-reversing-example-classic-Holder}
                  The direct method to use H\"{o}lder's inequality repeatedly.

          \end{enumerate}

          These proofs can be found in some literature listed in Table \ref{tab:classical-Young}.
          Terp indicated that one can prove \eqref{eq:Young-order-reversing-example-classic-state} by the method \ref{item:Young-order-reversing-example-classic-Holder} even when $G$ is not unimodular, but there is no explicit proof in this paper.
          By using \eqref{eq:Young-order-reversing-example-classic-state}, Terp generalized the Hausdorff--Young inequality to any locally compact group \cite[Theorem 5.2]{MR3730047}.
          The proof of Theorem \ref{thm:Young-order-reversing} in Section \ref{sec:Young-order-proof} can be regarded as a generalization of the method \ref{item:Young-order-reversing-example-classic-Holder}.

          \begin{table}
            \centering
            \caption{Some literature in which the proof of \eqref{eq:Young-order-reversing-example-classic-state} is mentioned}
            \label{tab:classical-Young}

            \vspace{8pt}

            \begin{tabular}{c||c|c}
                                                                            & $G$: unimodular                            & $G$: general                            \\
              \hline \hline
              \ref{item:Young-order-reversing-example-classic-Riesz-Thorin} & Weil \cite{MR0005741}                      & Klein--Russo \cite[Lemma 2.1]{MR499945} \\
              \hline
              \ref{item:Young-order-reversing-example-classic-Holder}       & Hewitt--Ross \cite[Theorem 20.18]{MR551496} & (Terp \cite[Lemma 1.1]{MR3730047})
            \end{tabular}
          \end{table}

    \item \label{item:Young-order-reversing-example-component}
          We have $Y ( p_1 , p_2 ; G ) \leq Y ( p_1 , p_2 ; G_0 )$ by Theorem \ref{thm:Young-order-reversing}, where $G_0 \subset G$ is the identity component.
          If $G_0$ is open (e.g. $G$ is a Lie group), then a Haar measure on $G_0$ corresponds to that on $G$ and hence $Y ( p_1 , p_2 ; G ) \geq Y ( p_1 , p_2 ; G_0 )$.
          Thus, we obtain $Y ( p_1 , p_2 ; G ) = Y ( p_1 , p_2 ; G_0 )$ when $G_0$ is open.

  \end{enumerate}

\end{example}

By using Theorem \ref{thm:Young-order-reversing}, we obtain necessary and sufficient conditions for satisfying $Y ( p_1 , p_2 ; G ) \leq Y ( p_1 , p_2 ; \mathbb{R} )$ as follows.

\begin{corollary}
  \label{cor:Young-R-compare}

  The following conditions \ref{item:Young-R-compare-open-compact}-\ref{item:Young-R-compare-not-1-G0} on the locally compact group $G$ are equivalent for any $1 < p_1 , p_2 < \infty$ with
  \begin{align}
    \frac{1}{p_1} + \frac{1}{p_2} > 1. \label{eq:Young-R-compare-condition}
  \end{align}

  \begin{enumerate}
    \item \label{item:Young-R-compare-open-compact}
          The locally compact group $G$ has no open compact subgroup.

    \item \label{item:Young-R-compare-G0-non-compact}
          The identity component $G_0 \subset G$ is not compact.

    \item \label{item:Young-R-compare-subgroup-G}
          The locally compact group $G$ has a closed subgroup which is isomorphic to $\mathbb{R}$ as a topological group.

    \item \label{item:Young-R-compare-subgroup-G0}
          The identity component $G_0$ has a closed subgroup which is isomorphic to $\mathbb{R}$ as a topological group.

    \item \label{item:Young-R-compare-less-G}
          One has $Y ( p_1 , p_2 ; G ) \leq Y ( p_1 , p_2 ; \mathbb{R} )$.

    \item \label{item:Young-R-compare-less-G0}
          One has $Y ( p_1 , p_2 ; G_0 ) \leq Y ( p_1 , p_2 ; \mathbb{R} )$.

    \item \label{item:Young-R-compare-not-1-G}
          One has $Y ( p_1 , p_2 ; G ) \neq 1$.

    \item \label{item:Young-R-compare-not-1-G0}
          One has $Y ( p_1 , p_2 ; G_0 ) \neq 1$.

  \end{enumerate}

\end{corollary}

When $G$ is unimodular, $\ref{item:Young-R-compare-open-compact} \Longleftrightarrow \ref{item:Young-R-compare-not-1-G}$ was proved by Fournier \cite[Theorems 1 and 3]{MR461034}, and $\ref{item:Young-R-compare-open-compact} \Longleftrightarrow \ref{item:Young-R-compare-G0-non-compact} \Longleftrightarrow \ref{item:Young-R-compare-less-G} \Longleftrightarrow \ref{item:Young-R-compare-less-G0} \Longleftrightarrow \ref{item:Young-R-compare-not-1-G} \Longleftrightarrow \ref{item:Young-R-compare-not-1-G0}$ was essentially proved in the previous paper of the author \cite[Corollary 1.3 and Remark 2.2]{MR4540843}.

When $G$ is not necessarily unimodular, $\ref{item:Young-R-compare-subgroup-G} \Longleftrightarrow \ref{item:Young-R-compare-subgroup-G0}$ follows from the connectedness of $\mathbb{R}$.
The results of Iwasawa \cite[Theorem 13]{MR29911} and the Gleason--Yamabe theorem \cite{MR39730} \cite[Theorem 5']{MR58607} imply $\ref{item:Young-R-compare-G0-non-compact} \Longrightarrow \ref{item:Young-R-compare-subgroup-G0}$.
Nielsen proved $\ref{item:Young-R-compare-open-compact} \Longleftrightarrow \ref{item:Young-R-compare-not-1-G}$ \cite[Theorem 1]{MR1304346}.
The author showed $\ref{item:Young-R-compare-open-compact} \Longleftrightarrow \ref{item:Young-R-compare-G0-non-compact}$ in the previous paper \cite[Remark 2.4 (3)]{MR4563397} by using the result of Hewitt--Ross \cite{MR551496}.
Theorem \ref{thm:Young-order-reversing} implies $\ref{item:Young-R-compare-subgroup-G0} \Longrightarrow \ref{item:Young-R-compare-less-G0} \Longrightarrow \ref{item:Young-R-compare-less-G}$.
Theorem \ref{thm:Young-order-reversing} and Example \ref{ex:Young-order-reversing-example} \ref{item:Young-order-reversing-example-classic} imply $\ref{item:Young-R-compare-not-1-G0} \Longrightarrow \ref{item:Young-R-compare-not-1-G}$.
The implications $\ref{item:Young-R-compare-less-G} \Longrightarrow \ref{item:Young-R-compare-not-1-G}$ and $\ref{item:Young-R-compare-less-G0} \Longrightarrow \ref{item:Young-R-compare-not-1-G0}$ are deduced from $Y ( p_1 , p_2 ; \mathbb{R} ) < 1$.
Beckner determined the value of $Y ( p_1 , p_2 ; \mathbb{R}^n )$ explicitly as the following fact and hence $Y ( p_1 , p_2 ; \mathbb{R} ) < 1$ is essentially obtained.

\begin{fact}[Beckner {\cite[Theorem 3]{MR385456}}]
  \label{fact:Beckner}

  Let $1 \leq p \leq \infty$ be as in \eqref{eq:optimal-constant-p-definition} for $1 \leq p_1 , p_2 \leq \infty$ with \eqref{eq:optimal-constant-p1-p2-condition}.
  Then
  \begin{align}
    Y ( p_1 , p_2 ; \mathbb{R}^n )
    = \left( \frac{B ( p_1 ) B ( p_2 )}{B ( p )} \right)^{n / 2}, &  &
    B ( p )
    :=
    \left\{
    \begin{aligned}
       & \frac{p^{1 / p}}{p'^{1 / p'}} &  & \text{if $1 < p < \infty$} \\
       & 1                             &  & \text{if $p = 1, \infty$}
    \end{aligned}
    \right.
  \end{align}
  holds for any $n \in \mathbb{Z}_{\geq 1}$.

\end{fact}

\begin{remark}

  There are some proofs of Fact \ref{fact:Beckner} (and its generalization named the Brascamp--Lieb inequality \cite[Theorem 1]{MR412366}).

  \begin{enumerate}
    \item
          Beckner reduced Fact \ref{fact:Beckner} to the problem of finding the integrable solutions of
          \begin{align}
            \psi_1 (x_1) \psi_2 (x_2) = \nu_1 (x_1 + x_2) \nu_2 (x_1 - x_2) \label{eq:independent}
          \end{align}
          by using the Minkowski integral inequality.
          Beckner proved the existence of rotation invariant extremal functions by using the Riesz--Sobolev rearrangement inequality, and proved Fact \ref{fact:Beckner} by showing that the rotation invariant solutions of \eqref{eq:independent} are only Gaussian functions.
          Lieb proved the Brascamp--Lieb inequality by showing implicitly that the solutions of \eqref{eq:independent} are only Gaussian functions (the Darmois--Skitovich theorem \cite{MR61322} \cite{MR0055597}) even when the function is not necessarily rotation invariant \cite[Theorem 6.2]{MR1069246}.

    \item
          Brascamp--Lieb proved Fact \ref{fact:Beckner} by showing $Y ( p_1 , p_2 ; \mathbb{R}^n ) = Y ( p_1 , p_2 ; \mathbb{R} )^n$ and bounding the behavior of $Y ( p_1 , p_2 ; \mathbb{R}^n )$ as $n \to \infty$ from above.
          In addition, Brascamp--Lieb generalized Fact \ref{fact:Beckner} to the Brascamp--Lieb inequality by a similar argument.

    \item
          Barthe gave a direct proof \cite[Theorem 1]{MR1616143} utilizing the change of variable by Henstock--Macbeath \cite[Section 5]{MR56669} and the weighted AM-GM inequality.
          In addition, Barthe proved that a similar argument is valid for the Brascamp--Lieb inequality \cite[Theorem 1]{MR1650312}.

    \item
          Carlen--Lieb--Loss proved the rank-one Brascamp--Lieb inequality by showing the property that the integral of the product of the exponentiations of the solutions of the heat equations is increasing with time \cite[Theorem 3.1]{MR2077162}.
          Cordero-Erausquin--Ledoux proved this property by using the estimate of Shannon's differential entropy \cite[Theorem 6]{MR2644890}.

  \end{enumerate}

  In addition, there are many works \cite{MR1008726} \cite{MR2377493} \cite{MR2448061} \cite{MR2496567} \cite{MR2661170} \cite{MR2674705} \cite{MR2806562} \cite{MR3364694} \cite{MR3239122} \cite{MR3431655} \cite{MR3723636} \cite{MR3610015} \cite{MR3777414} \cite{MR4173156} and surveys \cite{MR1898210} \cite{MR2657116} \cite{MR3204854} about Fact \ref{fact:Beckner} and the Brascamp--Lieb inequality.

\end{remark}

\subsection{Comparison with Corollary \ref{cor:Young-non-compact-dimension}}
\label{subsec:Young-known-result-non-compact-dimension}

In this section, we see some relations between known results and Corollary \ref{cor:Young-non-compact-dimension}.
Table \ref{tab:Young-non-compact-dimension-compare} lists the authors who proved Corollary \ref{cor:Young-non-compact-dimension} for some $G$.
If $G$ is compact, then $G$ is unimodular and Corollary \ref{cor:Young-non-compact-dimension} corresponds to \eqref{eq:Young-order-reversing-example-classic-state}.
Thus, Corollary \ref{cor:Young-non-compact-dimension} was essentially proved by Weil (Example \ref{ex:Young-order-reversing-example} \ref{item:Young-order-reversing-example-classic}).
The equivalent conditions \ref{item:Young-R-compare-open-compact}-\ref{item:Young-R-compare-not-1-G0} in Corollary \ref{cor:Young-R-compare} is also equivalent to $r ( G ) < \dim G$ for any connected Lie group $G$.
Thus, Corollary \ref{cor:Young-non-compact-dimension} gives a stronger bound than Corollary \ref{cor:Young-R-compare} of $Y ( p_1 , p_2 ; G )$ from above for any connected Lie group $G$ with $\# Z ( G / R ) < \infty$.

When $G$ is either a simply connected solvable Lie group or a connected nilpotent Lie group, Nielsen determined the value of $Y ( p_1 , p_2 ; G )$ as follows.

\begin{fact}[Nielsen {\cite[Corollaries (a) and (b)]{MR1304346}}]
  \label{fact:Nielsen}

  Suppose that $1 < p_1 , p_2 < \infty$ satisfy \eqref{eq:Young-R-compare-condition}.
  If $G$ is either a simply connected solvable Lie group or a connected nilpotent Lie group, then
  \begin{align}
    Y ( p_1 , p_2 ; G )
    = Y ( p_1 , p_2 ; \mathbb{R} )^{\dim G - \mathrm{rank} ( \ker ( \tilde{G} \to G ) ) }
  \end{align}
  holds, where $\tilde{G}$ is the universal cover of $G$.

\end{fact}

We have $r ( G ) = \mathrm{rank} ( \ker ( \tilde{G} \to G ) )$ for any connected solvable Lie group $G$ (Example \ref{ex:non-compact-dimension-bound-apply} \ref{item:non-compact-dimension-bound-apply-Nielsen}).
Thus, if $G$ is either a simply connected solvable Lie group or a connected nilpotent Lie group, then Corollary \ref{cor:Young-non-compact-dimension} follows from Fact \ref{fact:Nielsen}.
In addition, Bennett--Bez--Buschenhenke--Cowling--Flock gave a stronger bound than Corollary \ref{cor:Young-non-compact-dimension} in the case where each support of the functions is sufficiently small as follows.

\begin{table}
  \centering
  \caption{Authors who proved Corollary \ref{cor:Young-non-compact-dimension} for some $G$}
  \label{tab:Young-non-compact-dimension-compare}

  \vspace{8pt}

  \begin{tabular}{c|c}
    Connected Lie group $G$              & Author                                                                                                 \\
    \hline \hline
    Compact group                        & Weil (Example \ref{ex:Young-order-reversing-example} \ref{item:Young-order-reversing-example-classic}) \\
    \hline
    $\mathbb{R}^n$                       & Beckner (Fact \ref{fact:Beckner})                                                                      \\
    \hline
    Simply connected nilpotent Lie group & Klein--Russo \cite[Corollary 2.5']{MR499945}                                                           \\
    \hline
    Simply connected solvable Lie group, & \multirow{2}{*}{Nielsen (Fact \ref{fact:Nielsen})}                                                     \\
    Nilpotent Lie group                  &
  \end{tabular}
\end{table}

\begin{fact}[Bennett--Bez--Buschenhenke--Cowling--Flock {\cite[Corollary 2.4]{MR4173156}}]
  \label{fact:Bennett-Bez-Buschenhenke-Cowling-Flock}

  Let $1 \leq p \leq \infty$ be as in \eqref{eq:optimal-constant-p-definition} for $1 \leq p_1 , p_2 \leq \infty$ with \eqref{eq:optimal-constant-p1-p2-condition}.
  Then for any Lie group $G$ and any $Y_0 > Y ( p_1 , p_2 ; \mathbb{R} )^{\dim G}$, there exists a non-empty open set $V \subset G$ such that
  \begin{align}
    \| \phi_1 * ( \phi_2 \Delta^{1 / p_1'} ) \|_p \leq Y_0 \| \phi_1 \|_{p_1} \| \phi_2 \|_{p_2} \label{eq:Bennett-Bez-Buschenhenke-Cowling-Flock-inequality}
  \end{align}
  for any measurable functions $\phi_1 , \phi_2 \colon G \to \mathbb{C}$ whose supports are contained in $V$.

\end{fact}

The connectedness is assumed in the original paper of Bennett--Bez--Buschenhenke--Cowling--Flock.
Nevertheless, Fact \ref{fact:Bennett-Bez-Buschenhenke-Cowling-Flock} holds without connectedness because $V$ can be rearranged to satisfy $V \subset G_0$.
The value $Y ( p_1 , p_2 ; \mathbb{R} )^{\dim G}$ in Fact \ref{fact:Bennett-Bez-Buschenhenke-Cowling-Flock} is the best possible by the result of Cowling--Martini--M\"{u}ller--Parcet \cite[Proposition 2.4 (i)]{MR4000236}.

Fact \ref{fact:Bennett-Bez-Buschenhenke-Cowling-Flock} gives a stronger bound than Theorem \ref{thm:Young-order-reversing} with the assumption that the supports are sufficiently small.
For example, when $G = SL_2 ( \mathbb{R} )$, we have \eqref{eq:SL2-bound} by $\dim G = 3$ and $r ( G ) = \dim SO ( 2 ) = 1$.
On the other hand, for any $Y_0 > Y ( p_1 , p_2 ; \mathbb{R} )^3$, there exists a non-empty open set $V \subset SL_2 ( \mathbb{R} )$ such that \eqref{eq:Bennett-Bez-Buschenhenke-Cowling-Flock-inequality} holds for any $\phi_1 , \phi_2 \colon SL_2 ( \mathbb{R} ) \to \mathbb{C}$ whose supports are contained in $V$.

\section{Optimal constant on the boundary}
\label{sec:boundary}

In this section, we show Theorem \ref{thm:Young-order-reversing} when $( p_1 , p_2 )$ is on the boundary of the range that satisfies the assumption.
In this case, the equality of the classical Young's inequality \eqref{eq:Young-order-reversing-example-classic-state} holds for any locally compact group $G$ as follows.

\begin{lemma}
  \label{lem:Young-bound}

  Suppose that $1 \leq p_1 , p_2 \leq \infty$ with \eqref{eq:optimal-constant-p1-p2-condition} satisfy at least one of the following cases \ref{item:Young-bound-p1}, \ref{item:Young-bound-p2}, and \ref{item:Young-bound-p}.

  \begin{enumerate}
    \item \label{item:Young-bound-p1}
          One has $p_1 = 1$.
    \item \label{item:Young-bound-p2}
          One has $p_2 = 1$.
    \item \label{item:Young-bound-p}
          One has $1 / p_1 + 1 / p_2 = 1$.
  \end{enumerate}
  Then $Y ( p_1 , p_2 ; G ) = 1$ holds for any locally compact group $G$.

\end{lemma}

Theorem \ref{thm:Young-order-reversing} follows from Lemma \ref{lem:Young-bound} in the cases \ref{item:Young-bound-p1}, \ref{item:Young-bound-p2}, and \ref{item:Young-bound-p}.
We show the following lemma of the symmetry of the optimal constant to prove Lemma \ref{lem:Young-bound}.

\begin{lemma}
  \label{lem:optimal-constant-transform}

  Let $G$, $p_1$, $p_2$, $p$, $\Delta$, and $Y ( p_1 , p_2 ; G )$ be as in Definition \ref{def:optimal-constant}.

  \begin{enumerate}
    \item \label{item:optimal-constant-transform-convolution}
          One has
          \begin{align}
            \phi_1 * ( \phi_2 \Delta^{1 / p_1'} ) ( g ' )
            = \left( \frac{\phi_2 ( \cdot^{-1} )}{\Delta^{1 / p_2}} \right) * \left( \frac{\phi_1 ( \cdot^{-1} )}{\Delta^{1 / p}} \right) ( g'^{-1} ) \Delta ( g' )^{- 1 / p} \label{eq:optimal-constant-transform-convolution-state}
          \end{align}
          for any $g' \in G$ and any measurable functions $\phi_1 , \phi_2 \colon G \to \mathbb{C}$.

    \item \label{item:optimal-constant-transform-constant}
          One has $Y ( p_1 , p_2 ; G ) = Y ( p_2 , p_1 ; G )$.

  \end{enumerate}

\end{lemma}

\begin{proof}
  \begin{enumerate}
    \item
          We have
          \begin{align}
            \phi_1 * ( \phi_2 \Delta^{1 / p_1'} ) ( g ' )
            = \int_{G}^{} \phi_1 ( g' g ) \phi_2 ( g^{-1} ) \Delta ( g^{-1} )^{1 / p_1'} dg \label{eq:convolution-transform}
          \end{align}
          by the left invariance of $dg$.
          Since
          \begin{align}
            \frac{1}{p} + \frac{1}{p_1'}
            = \frac{1}{p_1} + \frac{1}{p_2} - 1 + \frac{1}{p_1'}
            = \frac{1}{p_2} \label{eq:p-definition-transform}
          \end{align}
          holds by \eqref{eq:optimal-constant-p-definition}, we obtain \eqref{eq:optimal-constant-transform-convolution-state} by \eqref{eq:convolution-transform}.

    \item
          It suffices to show
          \begin{align}
            \| \phi_1 * ( \phi_2 \Delta^{1 / p_1'} ) \|_p
            \leq Y ( p_2 , p_1 ; G) \| \phi_1 \|_{p_1} \| \phi_2 \|_{p_2} \label{eq:transform-conclusion}
          \end{align}
          for any measurable functions $\phi_1 , \phi_2 \colon G \to \mathbb{R}_{\geq 0}$.
          We have
          \begin{align}
            \| \phi_1 * ( \phi_2 \Delta^{1 / p_1'} ) \|_p^p
             & = \int_{G}^{} \left( \frac{\phi_2 ( \cdot^{-1} )}{\Delta^{1 / p_2}} \right) * \left( \frac{\phi_1 ( \cdot^{-1} )}{\Delta^{1 / p}} \right) ( g'^{-1} )^p \Delta ( g' )^{- 1} dg'        \\
             & = \left\| \left( \frac{\phi_2 ( \cdot^{-1} )}{\Delta^{1 / p_2}} \right) * \left( \frac{\phi_1 ( \cdot^{-1} )}{\Delta^{1 / p}} \right) \right\|_p^p \label{eq:convolution-Lp-transform}
          \end{align}
          by \eqref{eq:modular-function} and the assertion \ref{item:optimal-constant-transform-convolution}.
          Since
          \begin{align}
            \frac{1}{p} + \frac{1}{p_2'}
            = \frac{1}{p_1} \label{eq:p-definition-transform-similar}
          \end{align}
          holds by an argument similar to \eqref{eq:p-definition-transform}, we obtain \eqref{eq:transform-conclusion} by \eqref{eq:modular-function} and \eqref{eq:convolution-Lp-transform}.
          \qedhere
  \end{enumerate}

\end{proof}

\begin{proof}[Proof of Lemma \ref{lem:Young-bound}]
  It suffices to show
  \begin{align}
    Y ( p_1 , p_2 ; G) \geq 1 \label{eq:Young-classic-reverse}
  \end{align}
  by Example \ref{ex:Young-order-reversing-example} \ref{item:Young-order-reversing-example-classic}.

  First, we show \eqref{eq:Young-classic-reverse} in the case \ref{item:Young-bound-p}.
  We assume that an integrable function $\phi \colon G \to \mathbb{R}_{\geq 0}$ satisfies $\| \phi \|_1 = 1$ and
  \begin{align}
    \phi_1 ( g )
              & := \left\{
    \begin{aligned}
       & \phi ( g )^{1 / p_1} &  & \text{if $p_1 < \infty$} \\
       & 1                    &  & \text{if $p_1 = \infty$}
    \end{aligned}
    \right. , &
    \phi_2 ( g )
              & := \left\{
    \begin{aligned}
       & \left( \frac{\phi ( g^{-1} )}{\Delta ( g )} \right)^{1 / p_2} &  & \text{if $p_2 < \infty$} \\
       & 1                                                             &  & \text{if $p_2 = \infty$}
    \end{aligned}
    \right. .
  \end{align}
  Since
  \begin{align}
    \| \phi_1 \|_{p_1}
         & = \| \phi \|_1
    = 1, &
    \| \phi_2 \|_{p_2}
         & = \int_{G}^{} \frac{\phi ( g^{-1} )}{\Delta ( g )} dg
    = \| \phi \|_1
    = 1
  \end{align}
  hold by $1 / p_1 + 1 / p_2 = 1$ and \eqref{eq:modular-function}, the convolution $\phi_1 * ( \phi_2 \Delta^{1 / p_1'} )$ is continuous.
  Thus, one has $\| \phi_1 * ( \phi_2 \Delta^{1 / p_1'} ) \|_\infty = 1$ by
  \begin{align}
    \phi_1 * ( \phi_2 \Delta^{1 / p_1'} ) ( e )
    = \| \phi \|_1
    = 1
  \end{align}
  and hence we obtain \eqref{eq:Young-classic-reverse}.

  Second, we show \eqref{eq:Young-classic-reverse} in the case \ref{item:Young-bound-p1}.
  One may assume $p_2 < \infty$ by the case \ref{item:Young-bound-p}.
  When $\phi_2 \in L^{p_2} ( G )$ is fixed, there exists a sequence of integrable functions $\phi_{1 , m} \colon G \to \mathbb{R}_{\geq 0}$ with $\| \phi_{1 , m} \|_1 = 1$ such that $\phi_{1 , m} * \phi_2$ converges to $\phi_2$ \cite[Theorem 20.15]{MR551496}.
  Thus, we obtain \eqref{eq:Young-classic-reverse} because $\| \phi_{1 , m} * \phi_2 \|_{p_2}$ converges to $\| \phi_1 \|_{p_1} \| \phi_2 \|_{p_2} = \| \phi_2 \|_{p_2}$.

  Finally, \eqref{eq:Young-classic-reverse} in the case \ref{item:Young-bound-p2} follows from the case \ref{item:Young-bound-p1} by Lemma \ref{lem:non-compact-dimension-bound} \ref{item:optimal-constant-transform-constant}.
\end{proof}

\section{Proof of Theorem \ref{thm:Young-order-reversing}}
\label{sec:Young-order-proof}

In this section, we prove Theorem \ref{thm:Young-order-reversing} in the general case.
In Subsection \ref{subsec:Young-order-proof-Haar-subgroup}, we prepare a lemma (Lemma \ref{lem:subgroup-measure}) which represents the left Haar measure on $G$ by that on a closed subgroup $H \subset G$.
In Subsection \ref{subsec:Young-order-proof-Holder}, we give some inequalities (Example \ref{ex:Holder-apply}) to prove Theorem \ref{thm:Young-order-reversing} by applying H\"{o}lder's inequality (Fact \ref{fact:Holder}).
In Subsection \ref{subsec:Young-order-proof-Young-order-proof}, we complete the proof of Theorem \ref{thm:Young-order-reversing} by using Subsection \ref{subsec:Young-order-proof-Haar-subgroup} and Subsection \ref{subsec:Young-order-proof-Holder}.

\subsection{Representing the Haar measure by using a closed subgroup}
\label{subsec:Young-order-proof-Haar-subgroup}

In this subsection, we prove a lemma (Lemma \ref{lem:subgroup-measure}) which represents the left Haar measure on $G$ by that on a closed subgroup $H \subset G$, and we give some examples (Example \ref{ex:subgroup-represent}) which is used in the proof of Theorem \ref{thm:Young-order-reversing}.
We write $X := H \backslash G$ and $\overline{g} := H g \in X$ for the right coset of $g \in G$.

\begin{lemma}
  \label{lem:subgroup-measure}

  Let $H \subset G$ be a closed subgroup of a locally compact group $G$.

  \begin{enumerate}
    \item \label{item:subgroup-measure-modular-extend}
          There exists a continuous function $\delta \colon G \to \mathbb{R}_{> 0}$ such that $\delta |_H$ is the modular function of $H$ and
          \begin{align}
            \int_{H}^{} \phi ( h g ) dh \delta ( g ) \label{eq:subgroup-measure-modular-extend-invariant}
          \end{align}
          is left $H$-invariant for any measurable function $\phi \colon G \to \mathbb{C}$.

    \item \label{item:subgroup-measure-quotient}
          We fix $\delta$ in \ref{item:subgroup-measure-modular-extend}.
          Then there exists a Borel measure $d \overline{g}$ on $X$ such that
          \begin{align}
            \int_{X}^{} \int_{H}^{} \phi ( h g ) dh \delta ( g ) d \overline{g}
            = \int_{G}^{} \phi ( g ) dg \label{eq:subgroup-measure-quotient-condition}
          \end{align}
          for any integrable function $\phi \colon G \to \mathbb{C}$.

  \end{enumerate}

\end{lemma}

\begin{proof}

  \begin{enumerate}
    \item
          There exists a continuous function $\rho \colon G \to \mathbb{R}_{> 0}$ such that
          \begin{align}
            \rho ( g h ) = \frac{\Delta_H ( h ) \rho ( g )}{\Delta ( h )} \label{eq:rho-function}
          \end{align}
          for any $g \in G$ and $h \in H$ \cite[Proposition 2.56]{MR3444405}, where $\Delta_H$ is the modular function of $H$.
          By scaling, we may assume $\rho ( e ) = 1$.
          When the continuous function $\delta \colon G \to \mathbb{R}_{> 0}$ is defined as
          \begin{align}
            \delta ( g )
            := \frac{1}{\Delta ( g^{-1} ) \rho ( g^{-1})}, \label{eq:delta-definition}
          \end{align}
          we have
          \begin{align}
            \delta ( h g )
             = \frac{1}{\Delta ( g^{-1} h^{-1} ) \rho ( g^{-1} h^{-1} )}
             = \frac{\Delta_H ( h )}{\Delta ( g^{-1} ) \rho ( g^{-1})}                               
             = \Delta_H ( h ) \delta ( g ) \label{eq:delta-distribute}
          \end{align}
          by \eqref{eq:rho-function}.
          In particular, we have $\delta ( h ) = \Delta_H ( h ) \delta ( e ) = \Delta_H ( h )$.
          Thus, it follows from  \eqref{eq:delta-distribute} that
          \begin{align}
            \int_{H}^{} \phi ( h h' g ) dh \delta ( h' g )
            = \int_{H}^{} \phi ( h h' g ) dh \delta ( h' ) \delta ( g )
            = \int_{H}^{} \phi ( h g ) dh \delta ( g )
          \end{align}
          for any $h' \in H$ and any measurable function $\phi \colon G \to \mathbb{C}$ and hence \eqref{eq:subgroup-measure-modular-extend-invariant} is left $H$-invariant.

    \item
    For any $g \in G$ and $h' \in H$, we have
    \begin{align}
      \delta ( h' g ) \int_{H}^{} \phi ( h g ) dh
      = \delta ( g ) \int_{H}^{} \phi ( h h'^{-1} g ) dh
      = \delta ( g ) \delta ( h' ) \int_{H}^{} \phi ( h g ) dh
    \end{align}
    for any measurable function $\phi \colon G \to \mathbb{C}$ by the left invariance of \eqref{eq:subgroup-measure-modular-extend-invariant} and hence
    \begin{align}
      \delta ( h' g ) = \delta ( h' ) \delta ( g ). \label{eq:delta-distribute-obtain}
    \end{align}
    Since
    \begin{align}
      \rho ( g ) := \frac{1}{\Delta ( g ) \delta ( g^{-1} )}
    \end{align}
    is a rho-function by \eqref{eq:delta-distribute-obtain}, there exists a Borel measure $dgH$ on $G / H$ such that
    \begin{align}
      \int_{G / H}^{} \int_{H}^{} \omega ( g h ) dh dgH
      = \int_{G}^{} \omega ( g ) \rho ( g ) dg \label{eq:dgH-condition}
    \end{align}
    for any integrable function $\omega \colon G \to \mathbb{C}$ \cite[Theorem 2.58]{MR3444405}.
By \eqref{eq:modular-function} and applying 
\begin{align}
  \omega ( g ) = \frac{\phi ( g^{-1} )}{ \Delta ( g ) \rho ( g ) }
  = \phi ( g^{-1} ) \delta ( g^{ - 1 } )
\end{align}
to \eqref{eq:dgH-condition}, we have
\begin{align}
  \int_{G}^{} \phi ( g ) dg
  = \int_{G / H}^{} \int_{H}^{} \phi ( ( g h )^{-1} ) \delta ( ( g h )^{-1} ) dh dgH. \label{eq:phi-integral}
\end{align}
Since $\delta |_H$ is the modular function of $H$, we have
\begin{align}
  \int_{H}^{} \phi ( ( g h )^{-1} ) \delta ( ( g h )^{-1} ) dh
  = \int_{H}^{} \phi ( h g^{-1} ) dh \delta ( g^{-1} ) \label{eq:dgH-condition-apply}
\end{align}
by \eqref{eq:modular-function} and \eqref{eq:delta-distribute-obtain}.
Thus, by setting the Borel measure $d \overline{g}$ on $X := H \backslash G$ as
\begin{align}
  \int_{X}^{} \alpha ( \overline{g} ) d\overline{g}
  = \int_{G / H}^{} \alpha ( H g^{-1} ) dgH,
\end{align}
we obtain \eqref{eq:subgroup-measure-quotient-condition} by \eqref{eq:phi-integral} and \eqref{eq:dgH-condition-apply}.
          \qedhere

  \end{enumerate}

\end{proof}

Now we give some examples of Lemma \ref{lem:subgroup-measure} to prove Theorem \ref{thm:Young-order-reversing}.

\begin{example}
  \label{ex:subgroup-represent}

  Let $\phi_1 , \phi_2 \colon G \to \mathbb{R}_{\geq 0}$ be measurable functions.
  We suppose $1 < p_1 , p_2 , p < \infty$ satisfy \eqref{eq:optimal-constant-p-definition}.

  \begin{enumerate}
    \item \label{item:subgroup-represent-phi1}
          Let $s ( h , g ) := \phi_1 ( h g ) \delta ( g )^{1 / p_1}$ for $g \in G$ and $h \in H$.
          Then
          \begin{align}
            S ( \overline{g} )
            := \int_{H}^{} s ( h , g )^{p_1} dh
            = \int_{H}^{} \phi_1 ( h g )^{p_1} dh \delta ( g )
          \end{align}
          is well-defined by Lemma \ref{lem:subgroup-measure} \ref{item:subgroup-measure-modular-extend}, and we have
          \begin{align}
            \int_{X}^{} S ( \overline{g} ) d\overline{g}
            = \| \phi_1 \|_{p_1}^{p_1} \label{eq:subgroup-represent-phi1-integral}
          \end{align}
          by Lemma \ref{lem:subgroup-measure} \ref{item:subgroup-measure-quotient}.

    \item \label{item:subgroup-represent-phi2-g'}
          Let $t ( g , g' ) := ( \phi_2 ( g^{-1} g' )^{p_2} \delta ( g' ) )^{1 / p}$ for $g , g' \in G$.
          Since
          \begin{align}
            \int_{H}^{} t ( h^{-1} h' g , g' )^p dh
            = \int_{H}^{} t ( h^{-1} g , g' )^p dh
          \end{align}
          for any $h' \in H$, the function
          \begin{align}
            T ( \overline{g} , \overline{g'} )
            := \int_{H}^{} t ( h^{-1} g , g' )^p dh
            = \int_{H}^{} \phi_2 ( g^{-1} h g' )^{p_2} dh \delta ( g' )
          \end{align}
          is well-defined by Lemma \ref{lem:subgroup-measure} \ref{item:subgroup-measure-modular-extend}.
          Thus, we have
          \begin{align}
            \int_{X}^{} T ( \overline{g} , \overline{g'} ) d\overline{g'}
            = \int_{X}^{} \int_{H}^{} \phi_2 ( g^{-1} h g' )^{p_2} dh \delta ( g' ) d\overline{g'}
            = \int_{G}^{} \phi_2 ( g^{-1} g' )^{p_2} dg'
            = \| \phi_2 \|_{p_2}^{p_2}
          \end{align}
          by lemma \ref{lem:subgroup-measure} \ref{item:subgroup-measure-quotient}.

    \item \label{item:subgroup-represent-phi2-g}
          Let
          \begin{align}
            u ( g , h , g' )
            := \left( \frac{\phi_2 ( g^{-1} h g' )^{p_2} \Delta ( g^{-1} h g' ) \delta ( g )}{\delta ( h )}  \right)^{1 / p_1'}
          \end{align}
          for $g , g' \in G$ and $h \in H$.
          Since $\delta |_H$ is the modular function of $H$ by Lemma \ref{lem:subgroup-measure} \ref{item:subgroup-measure-modular-extend},
          \begin{align}
            \int_{H}^{} u ( g , h , h' g' )^{p_1'} dh
             & = \int_{H}^{} \frac{\phi_2 ( g^{-1} h h' g' )^{p_2} \Delta ( g^{-1} h h' g' ) \delta ( g )}{\delta ( h )} dh                               \\
             & = \int_{H}^{} \phi_2 ( g^{-1} h^{-1} h' g' )^{p_2} \Delta ( g^{-1} h^{-1} h' g' ) dh \delta ( g )                                          \\
             & = \int_{H}^{} \phi_2 ( g^{-1} h^{-1} g' )^{p_2} \Delta ( g^{-1} h^{-1} g' ) dh \delta ( g ) \label{eq:subgroup-represent-phi2-g-invariant}
          \end{align}
          is independent of $h' \in H$ by \eqref{eq:modular-function}.
          Thus, the function
          \begin{align}
            U ( \overline{g} , \overline{g'} ) := \int_{H}^{} u ( g , h , g' )^{p_1'} dh
          \end{align}
          is well-defined by Lemma \ref{lem:subgroup-measure} \ref{item:subgroup-measure-modular-extend}.
          Now we prove
          \begin{align}
            \int_{X}^{} U ( \overline{g} , \overline{g'} ) d\overline{g}
            = \| \phi_2 \|_{p_2}^{p_2} \label{eq:subgroup-represent-phi2-g-integral}
          \end{align}
          for any $g' \in G$.
          We have
          \begin{align}
            \int_{X}^{} U ( \overline{g} , \overline{g'} ) d\overline{g}
            = \int_{X}^{} \int_{H}^{} \phi_2 ( g^{-1} h^{-1} g' )^{p_2} \Delta ( g^{-1} h^{-1} g' ) dh \delta ( g ) d\overline{g}
          \end{align}
          by \eqref{eq:subgroup-represent-phi2-g-invariant} and
          \begin{align}
            \int_{X}^{} \int_{H}^{} \phi_2 ( g^{-1} h^{-1} g' )^{p_2} \Delta ( g^{-1} h^{-1} g' ) dh \delta ( g ) d\overline{g}
             & = \int_{G}^{} \phi_2 ( g^{-1} g' )^{p_2} \Delta ( g^{-1} g' ) dg \\
             & = \int_{G}^{} \phi_2 ( g^{-1} )^{p_2} \Delta ( g^{-1} ) dg
          \end{align}
          by Lemma \ref{lem:subgroup-measure} \ref{item:subgroup-measure-quotient}.
          Since
          \begin{align}
            \int_{G}^{} \phi_2 ( g^{-1} )^{p_2} \Delta ( g^{-1} ) dg
            = \int_{G}^{} \phi_2 ( g )^{p_2} dg
            = \| \phi_2 \|_{p_2}^{p_2}
          \end{align}
          by \eqref{eq:modular-function}, we obtain \eqref{eq:subgroup-represent-phi2-g-integral}.

    \item \label{item:subgroup-represent-convolution}
          Let $s$, $t$, and $u$ be as in \ref{item:subgroup-represent-phi1}, \ref{item:subgroup-represent-phi2-g'}, and \ref{item:subgroup-represent-phi2-g}, respectively.
          Now we show
          \begin{align}
             & \quad \phi_1 * ( \phi_2 \Delta^{1/p_1'} ) ( h' g' )                                                                                                                                                                 \\
             & = \int_{X}^{} \int_{H}^{} \frac{s ( h , g ) t ( h'^{-1} h g , g') u ( g , h^{-1} h' , g' ) \delta ( h^{-1} h' )^{1 / p_1'}}{\delta ( g' )^{1 / p}} dh d\overline{g} \label{eq:subgroup-represent-convolution-claim}
          \end{align}
          for any $h' \in H$ and $g' \in G$.
          Since
          \begin{align}
             & \quad \phi_1 * ( \phi_2 \Delta^{1/p_1'} ) ( h' g' )                                                                                             \\
             & = \int_{G}^{} \phi_1 ( g ) \phi_2 ( g^{-1} h' g' ) \Delta ( g^{-1} h' g' )^{1 / p_1'} dg                                                        \\
             & = \int_{X}^{} \int_{H}^{} \phi_1 ( h g ) \phi_2 ( g^{-1} h^{-1} h' g' ) \Delta ( g^{-1} h^{-1} h' g' )^{1 / p_1'} dh \delta ( g ) d\overline{g}
          \end{align}
          by Lemma \ref{lem:subgroup-measure}, we obtain \eqref{eq:subgroup-represent-convolution-claim} by
          \begin{align}
             & \quad \phi_1 ( h g ) \phi_2 ( g^{-1} h^{-1} h' g' ) \Delta ( g^{-1} h^{-1} h' g' )^{1 / p_1'} \delta ( g )                  \\
             & = \frac{s ( h , g ) t ( h'^{-1} h g , g') u ( g , h^{-1} h' , g' ) \delta ( h^{-1} h' )^{1 / p_1'}}{\delta ( g' )^{1 / p}}.
          \end{align}

    \item \label{item:subgroup-represent-Lp-norm}
          Since
          \begin{align}
            \| \phi_1 * ( \phi_2 \Delta^{1/p_1'} ) \|_p^p
            = \int_{X}^{} \int_{H}^{} \phi_1 * ( \phi_2 \Delta^{1/p_1'} ) ( h' g' )^p dh' \delta ( g' ) d\overline{g'}
          \end{align}
          by Lemma \ref{lem:subgroup-measure}, we have
          \begin{align}
             & \quad \| \phi_1 * ( \phi_2 \Delta^{1/p_1'} ) \|_p^p                                                                                                                                \\
             & = \int_{X}^{} \int_{H}^{} \left( \int_{X}^{} \int_{H}^{} s ( h , g ) t ( h'^{-1} h g , g') u ( g , h^{-1} h' , g' ) \delta ( h^{-1} h' )^{1 / p_1'} dh d\overline{g} \right)^p dh' d\overline{g'}
          \end{align}
          by \ref{item:subgroup-represent-convolution}.

  \end{enumerate}

\end{example}

\subsection{H\"{o}lder's inequality}
\label{subsec:Young-order-proof-Holder}

In this subsection, we obtain some inequalities (Example \ref{ex:Holder-apply}) by using H\"{o}lder's inequality (Fact \ref{fact:Holder}) to prove Theorem \ref{thm:Young-order-reversing}.

\begin{fact}[H\"{o}lder's inequality]
  \label{fact:Holder}

  Let $k , l \in \mathbb{Z}_{\geq 1}$ and $p_{i,j} , c_i > 0$ for $i = 1 , \cdots , k$ and $j = 1 , \cdots , l$.
  Then
  \begin{align}
     & \quad \left( \int_{G}^{} \phi_1 (g)^{p_1} \cdots \phi_l (g)^{p_l} dg \right)^c                                                                                         \\
     & \leq \left( \int_{G}^{} \phi_1 (g)^{p_{1,1}} \cdots \phi_l (g)^{p_{1,l}} dg \right)^{c_1} \cdots \left( \int_{G}^{} \phi_1 (g)^{p_{k,1}} \cdots \phi_l (g)^{p_{k,l}} dg \right)^{c_k}
  \end{align}
  holds for any measurable functions $\phi_1 , \cdots , \phi_l \colon G \to \mathbb{R}_{\geq 0}$ on a measure space $G$, where
  \begin{align}
    c := c_1 + \cdots + c_k, &  &
    p_j := \frac{p_{1,j} c_1 + \cdots + p_{k,j} c_k}{c}
  \end{align}
  for $j = 1 , \cdots , l$.

\end{fact}

\begin{example}
  \label{ex:Holder-apply}

  Let $\phi_1$, $\phi_2$, $S$, $t$, $T$, $u$, and $U$ be as in Example \ref{ex:subgroup-represent}.
  Suppose that $1 < p_1 , p_2 , p < \infty$ satisfy \eqref{eq:optimal-constant-p-definition}.

  \begin{enumerate}
    \item \label{item:Holder-apply-H}
          We have
          \begin{align}
            \left( \int_{H}^{} ( t ( h^{-1} g , g' ) u ( g , h , g' ) )^{p_2} dh \right)^{1 / p_2}
            \leq T ( \overline{g} , \overline{g'} )^{1 / p} U ( \overline{g} , \overline{g'} )^{1 / p_1'}
          \end{align}
          by \eqref{eq:p-definition-transform} and Fact \ref{fact:Holder}.

    \item \label{item:Holder-apply-X}
          We show
          \begin{align}
             & \quad \left( \int_{X}^{} S ( \overline{g} )^{1 / p_1} T ( \overline{g} , \overline{g'} )^{1 / p} U ( \overline{g} , \overline{g'} )^{1 / p_1'} d\overline{g} \right)^p                                                \\
             & \leq \left( \| \phi_1 \|_{p_1}^{p_1 / p_2'} \| \phi_2 \|_{p_2}^{p_2 / p_1'} \right)^p \int_{X}^{} S ( \overline{g} ) T ( \overline{g} , \overline{g'} ) d\overline{g}. \label{eq:Holder-apply-X-claim}
          \end{align}
          One has
          \begin{align}
            p \left( \frac{1}{p_1'} + \frac{1}{p_2'} \right) + 1
            = p \left( 2 - \frac{1}{p_1} - \frac{1}{p_2} + \frac{1}{p}\right)
            = p
          \end{align}
          by \eqref{eq:optimal-constant-p-definition}.
          Thus, it follows from \eqref{eq:p-definition-transform-similar} and Fact \ref{fact:Holder} that
          \begin{align}
             & \quad \left( \int_{X}^{} S ( \overline{g} )^{1 / p_1} T ( \overline{g} , \overline{g'} )^{1 / p} U ( \overline{g} , \overline{g'} )^{1 / p_1'} d\overline{g} \right)^p                                                                               \\
             & \leq \left( \int_{X}^{} S ( \overline{g} ) d\overline{g}^{1 / p_2'} \int_{X}^{} U ( \overline{g} , \overline{g'} ) d\overline{g}^{1 / p_1'} \right)^p \int_{X}^{} S ( \overline{g} ) T ( \overline{g} , \overline{g'} ) d\overline{g}
          \end{align}
          and hence we obtain \eqref{eq:Holder-apply-X-claim} by Example \ref{ex:subgroup-represent} \ref{item:subgroup-represent-phi1} and \ref{item:subgroup-represent-phi2-g}.

  \end{enumerate}

\end{example}

\subsection{Completion of the proof}
\label{subsec:Young-order-proof-Young-order-proof}

In this subsection, we complete the proof of Theorem \ref{thm:Young-order-reversing} by using Example \ref{ex:subgroup-represent} and Example \ref{ex:Holder-apply}.

\begin{proof}[Proof of Theorem \ref{thm:Young-order-reversing}]

  It suffices to consider the case of $1 < p_1 , p_2 , p < \infty$ by Lemma \ref{lem:Young-bound}.
  Let $s$, $S$, $t$, $T$, $u$, and $U$ be as in Example \ref{ex:subgroup-represent} for $\phi_1, \phi_2 \colon G \to \mathbb{R}_{\geq 0}$ with $\| \phi_1 \|_{p_1} = \| \phi_2 \|_{p_2} = 1$.
  Then it suffices to show
  \begin{align}
     \int_{X}^{} \int_{H}^{} \left( \int_{X}^{} F ( g , h' , g' ) d\overline{g}\right)^p dh' d\overline{g'}
     \leq Y ( p_1 , p_2 ; H )^p \label{eq:conclusion}
  \end{align}
  by Example \ref{ex:subgroup-represent} \ref{item:subgroup-represent-Lp-norm}, where 
  \begin{align}
    F ( g , h' , g' )
    := \int_{H}^{} s ( h , g ) t ( h'^{-1} h g , g') u ( g , h^{-1} h' , g' ) \delta ( h^{-1} h' )^{1 / p_1'} dh.
  \end{align}
  We have
  \begin{align}
     \int_{H}^{} \left( \int_{X}^{} F ( g , h' , g' ) d\overline{g} \right)^p dh'
     \leq \left( \int_{X}^{} \left( \int_{H}^{} F ( g , h' , g' )^p dh' \right)^{1 / p} d\overline{g} \right)^p \label{eq:Minkowski-apply}
  \end{align}
  by the Minkowski integral inequality.
  Since $\delta |_H$ is the modular function of $H$ by Lemma \ref{lem:subgroup-measure} \ref{item:subgroup-measure-modular-extend}, we have
  \begin{align}
     \left( \int_{H}^{} F ( g , h' , g' )^p dh' \right)^{1 / p}
     & \leq Y ( p_1 , p_2 ; H ) S ( \overline{g} )^{1 / p_1} \left( \int_{H}^{} ( t ( h^{-1} g , g' ) u ( g , h , g' ) )^{p_2} dh \right)^{1 / p_2}                  \\
     & \leq Y ( p_1 , p_2 ; H ) S ( \overline{g} )^{1 / p_1} T ( \overline{g} , \overline{g'} )^{1 / p} U ( \overline{g} , \overline{g'} )^{1 / p_1'}
  \end{align}
  by Example \ref{ex:Holder-apply} \ref{item:Holder-apply-H}.
  Thus, it follows from \eqref{eq:Minkowski-apply} that
  \begin{align}
     \int_{H}^{} \left( \int_{X}^{} F ( g , h' , g' ) d\overline{g} \right)^p dh'
     \leq \left( Y ( p_1 , p_2 ; H ) \int_{X}^{} S ( \overline{g} )^{1 / p_1} T ( \overline{g} , \overline{g'} )^{1 / p} U ( \overline{g} , \overline{g'} )^{1 / p_1'} d\overline{g}  \right)^p.
  \end{align}
  Since
  \begin{align}
    \left( \int_{X}^{} S ( \overline{g} )^{1 / p_1} T ( \overline{g} , \overline{g'} )^{1 / p} U ( \overline{g} , \overline{g'} )^{1 / p_1'} d\overline{g} \right)^p
    \leq \int_{X}^{} S ( \overline{g} ) T ( \overline{g} , \overline{g'} ) d\overline{g}
  \end{align}
  holds by $\| \phi_1 \|_{p_1} = \| \phi_2 \|_{p_2} = 1$ and Example \ref{ex:Holder-apply} \ref{item:Holder-apply-X}, we have
  \begin{align}
     \int_{H}^{} \left( \int_{X}^{} F ( g , h' , g' ) d\overline{g} \right)^p dh'
     \leq Y ( p_1 , p_2 ; H )^p \int_{X}^{} S ( \overline{g} ) T ( \overline{g} , \overline{g'} ) d\overline{g}.
  \end{align}
  Thus, it follows that
  \begin{align}
     \int_{X}^{} \int_{H}^{} \left( \int_{X}^{} F ( g , h' , g' ) d\overline{g} \right)^p dh' d\overline{g'}
     \leq Y ( p_1 , p_2 ; H )^p \int_{X}^{} \int_{X}^{} S ( \overline{g} ) T ( \overline{g} , \overline{g'} ) d\overline{g} d\overline{g'}.
  \end{align}
  Since
  \begin{align}
    \int_{X}^{} \int_{X}^{} S ( \overline{g} ) T ( \overline{g} , \overline{g'} ) d\overline{g} d\overline{g'}
    = \int_{X}^{} \int_{X}^{} T ( \overline{g} , \overline{g'} ) d\overline{g'} S ( \overline{g} ) d\overline{g}
    = \int_{X}^{} S ( \overline{g} ) d\overline{g}
    = 1
  \end{align}
  by $\| \phi_1 \|_{p_1} = \| \phi_2 \|_{p_2} = 1$ and Example \ref{ex:subgroup-represent} \ref{item:subgroup-represent-phi1} \ref{item:subgroup-represent-phi2-g'}, we obtain \eqref{eq:conclusion}.
\end{proof}

\section{Proof of Corollary \ref{cor:Young-non-compact-dimension}}
\label{sec:non-compact-dimension}

In this section, we show Corollary \ref{cor:Young-non-compact-dimension} by using Theorem \ref{thm:Young-order-reversing} and the argument of Jing--Tran--Zhang \cite{MR4616694}.
We write $\mathcal{A}$ for the set of the connected Lie groups $G$ satisfying the assumption in Corollary \ref{cor:Young-non-compact-dimension} (i.e. the center of the semisimple part of $G$ is a finite group).
We note that any connected solvable Lie group is an element of $\mathcal{A}$.
Let $r ( G )$ be the dimension of the maximal compact subgroups for a connected Lie group $G$.
We show the following lemma to prove Corollary \ref{cor:Young-non-compact-dimension}.

\begin{lemma}
  \label{lem:non-compact-dimension-bound}

  Suppose that a real number $d ( G )$ is defined for each $G \in \mathcal{A}$, and
  \begin{align}
    d ( G )
    \geq d ( H ) + d ( G / H ) \label{eq:non-compact-dimension-bound-normal}
  \end{align}
  holds for any connected closed normal subgroup $H \in \mathcal{A}$ of any $G \in \mathcal{A}$ with $G / H \in \mathcal{A}$.
  For $G \in \mathcal{A}$, we denote by $I ( G )$ the inequality
  \begin{align}
    d ( G )
    \geq d ( \mathbb{R} ) \dim G + ( d ( \mathbb{R} / \mathbb{Z} ) - d ( \mathbb{R} ) ) r ( G ).
  \end{align}

  \begin{enumerate}
    \item \label{item:non-compact-dimension-bound-exact}
          Suppose that a connected closed normal subgroup $H \in \mathcal{A}$ of $G \in \mathcal{A}$ satisfies $G / H \in \mathcal{A}$.
          If $I ( H )$ and $I ( G / H )$ hold, then $I ( G )$ also holds.

    \item \label{item:non-compact-dimension-bound-solvable}
          Every non-trivial connected solvable Lie group $G$ satisfies $I ( G )$.

    \item \label{item:non-compact-dimension-bound-order}
          Furthermore, we assume that
          \begin{align}
            d ( \mathbb{R} / \mathbb{Z} ) = 0 \leq d ( H ) \leq d ( G ) \label{eq:non-compact-dimension-bound-order-assume}
          \end{align}
          for any connected closed subgroup $H \in \mathcal{A}$ of any $G \in \mathcal{A}$.
          Then every $G \in \mathcal{A}$ satisfies $I ( G )$, that is,
          \begin{align}
            d ( G )
            \geq d ( \mathbb{R} ) ( \dim G - r ( G ) ). \label{eq:non-compact-dimension-bound-order-simple}
          \end{align}

  \end{enumerate}

\end{lemma}

Jing--Tran--Zhang generalized the Brunn--Minkowski inequality to any Lie group by essentially using Lemma \ref{lem:non-compact-dimension-bound} \cite[Theorem 1.1]{MR4616694}.
Now we give some examples of Lemma \ref{lem:non-compact-dimension-bound} \ref{item:non-compact-dimension-bound-solvable}.

\begin{example}
  \label{ex:non-compact-dimension-bound-apply}

  \begin{enumerate}
    \item \label{item:non-compact-dimension-bound-apply-equal}
          If the equality of \eqref{eq:non-compact-dimension-bound-normal} holds for any connected closed normal subgroup $H \in \mathcal{A}$ of any $G \in \mathcal{A}$, then the equality of $I ( G )$ also holds for any connected solvable Lie group $G$.
          Actually, if $G$ is not trivial, then this claim can be shown by replacing $d ( G )$ with $- d ( G )$ and applying Lemma \ref{lem:non-compact-dimension-bound} \ref{item:non-compact-dimension-bound-solvable}.
          If $G$ is trivial, we have $d ( G ) = d ( G ) + d ( G )$ because the equality of \eqref{eq:non-compact-dimension-bound-normal} holds.
          Thus, we have $d ( G ) = r ( G ) = 0$ and hence the equality of $I ( G )$ holds.

    \item \label{item:non-compact-dimension-bound-apply-Nielsen}
          Let $d ( G ) := \mathrm{rank} ( \ker ( \tilde{G} \to G ) )$ for connected Lie group $G$, where $\tilde{G}$ is the universal covering of $G$.
          Now we show $d ( G ) = r ( G )$ for any connected solvable Lie group $G$.
          The kernel $\ker ( \tilde{G} \to G )$ is isomorphic to the fundamental group $\pi_1 ( G )$ of $G$ \cite[Theorem 9.5.4]{MR3025417}.
          Since $\pi_1 ( H ) \to \pi_1 ( G )$ is injective and $\pi_1 ( G ) / \pi_1 ( H )$ is isomorphic to $\pi_1 ( G / H )$ for any $G \in \mathcal{A}$ and any connected closed normal subgroup $H \in \mathcal{A}$ of $G$ \cite[Remark 11.1.17]{MR3025417}, we have
          \begin{align}
            d ( G )
            = \mathrm{rank} ( \pi_1 ( G ) )
            = \mathrm{rank} ( \pi_1 ( H ) ) + \mathrm{rank} ( \pi_1 ( G / H ) )
            = d ( H ) + d ( G / H ).
          \end{align}
          Thus, the equality of $I ( G )$ holds for any connected solvable Lie group by \ref{item:non-compact-dimension-bound-apply-equal}.
          In this case, we obtain
          \begin{align}
            d ( G )
            = d ( \mathbb{R} ) \dim G + ( d ( \mathbb{R} / \mathbb{Z} ) - d ( \mathbb{R} ) ) r ( G )
            = r ( G )
          \end{align}
          by $d ( \mathbb{R} ) = 0$ and $d ( \mathbb{R} / \mathbb{Z} ) = 1$.

  \end{enumerate}

\end{example}

Now we show the following lemma to prove Lemma \ref{lem:non-compact-dimension-bound}.

\begin{lemma}
  \label{lem:non-compact-dimension-deduce}

  Let $G$ be a connected Lie group.

  \begin{enumerate}
    \item \label{item:non-compact-dimension-deduce-normal}
          One has $r ( G ) = r ( H ) + r ( G / H )$ for any closed normal subgroup $H \lhd G$.

    \item \label{item:non-compact-dimension-deduce-exist-solvable}
          Every solvable Lie group $G$ with $\dim G \geq 2$ satisfies the following condition \ref{item:non-compact-dimension-deduce-exist-solvable-condition}.

          \begin{enumerate}
            \item \label{item:non-compact-dimension-deduce-exist-solvable-condition}
                  There exists a closed normal subgroup $H \in \mathcal{A}$ of $G$ such that $G / H \in \mathcal{A}$ and $1 \leq \dim H < \dim G$.
          \end{enumerate}

    \item \label{item:non-compact-dimension-deduce-exist}
          Suppose that $G \in \mathcal{A}$ and $\dim G \geq 2$.
          If $G$ does not satisfy the condition \ref{item:non-compact-dimension-deduce-exist-solvable-condition}, then there exists a closed subgroup $H \in \mathcal{A}$ of $G$ with $\dim H < \dim G$ such that
          \begin{align}
            \dim H - r ( H ) = \dim G - r ( G ). \label{eq:non-compact-dimension-deduce-exist-state}
          \end{align}

  \end{enumerate}

\end{lemma}

\begin{proof}

  \begin{enumerate}
    \item
          Let $K \subset G$ be a maximal compact subgroup of $G$.
          Then $K \cap H$ and $K / (K \cap H)$ are maximal compact subgroups of $H$ and $G / H$, respectively \cite[Theorem 14.3.13 (i) (a)]{MR3025417}.
          Thus, we obtain
          \begin{align}
            r ( G )
            = \dim K
            = \dim ( K \cap H ) + \dim ( K / (K \cap H) )
            = r ( H ) + r ( G / H ).
          \end{align}

    \item
          Since $G$ is a connected solvable Lie group, there exists a connected closed solvable normal subgroup $H \lhd G$ such that $\dim ( G / H ) = 1$.
          Then $1 \leq \dim H < \dim G$ holds by $\dim G \geq 2$, and we have
          \begin{align}
            \dim G
            = \dim H + \dim ( G / H ). \label{eq:dimension-sum}
          \end{align}
          Since $G / H$ is abelian by $\dim ( G / H ) = 1$, we have $G / H \in \mathcal{A}$.
          Thus, $G$ satisfies the condition \ref{item:non-compact-dimension-deduce-exist-solvable-condition}.

    \item
          Let $R \lhd G$ be the radical (the largest connected solvable closed normal subgroup) of $G$.
          Since $\dim G \geq 2$ and $G$ does not satisfy the condition \ref{item:non-compact-dimension-deduce-exist-solvable-condition}, $G$ is not a solvable Lie group by \ref{item:non-compact-dimension-deduce-exist-solvable}.
          Thus, we have $\dim R < \dim G$.
          Since $G / R \in \mathcal{A}$ by $G \in \mathcal{A}$, the connected Lie group $G$ is semisimple (i.e. $\dim R = 0$) by the assumption that $G$ does not satisfy the condition \ref{item:non-compact-dimension-deduce-exist-solvable-condition}.
          Let $G = K A N$ be the Iwasawa decomposition.
          Then the closed subgroup $H := A N \subset G$ is a simply connected solvable Lie group \cite[Theorem 6.46]{MR1920389}.
          Thus, we have $\dim H < \dim G$, $H \in \mathcal{A}$, and $r ( H ) = 0$.
          Since $K \subset G$ is a maximal compact subgroup of $G$ by $G \in \mathcal{A}$ \cite[Theorem 6.31 (g)]{MR1920389}, the equality
          \begin{align}
            \dim H - r ( H )
            = \dim G - \dim K
            = \dim G - r ( G )
          \end{align}
          is obtained.
          \qedhere

  \end{enumerate}

\end{proof}

\begin{proof}[Proof of Lemma \ref{lem:non-compact-dimension-bound}]

  \begin{enumerate}
    \item
          Since
          \begin{align}
            d ( G )
            \geq d ( \mathbb{R} ) ( \dim H + \dim ( G / H ) ) + ( d ( \mathbb{R} / \mathbb{Z} ) - d ( \mathbb{R} ) ) ( r ( H ) + r ( G / H ) )
          \end{align}
          by $I ( H )$, $I ( G / H )$, and \eqref{eq:non-compact-dimension-bound-normal}, we obtain $I ( G )$ by \eqref{eq:dimension-sum} and Lemma \ref{lem:non-compact-dimension-deduce} \ref{item:non-compact-dimension-deduce-normal}.

    \item
          We prove it by induction on $\dim G$.
          If $\dim G = 1$, then either $G = \mathbb{R}$ or $G = \mathbb{R} / \mathbb{Z}$ holds.
          We obtain $I ( \mathbb{R} )$ by $r ( \mathbb{R} ) = 0$, and $I ( \mathbb{R} / \mathbb{Z} )$ by $r ( \mathbb{R} / \mathbb{Z} ) = 1$.

          Now we show $I ( G )$ when $\dim G \geq 2$.
          Then $G$ is a solvable Lie group and hence $G$ satisfies the condition \ref{item:non-compact-dimension-deduce-exist-solvable-condition} by Lemma \ref{lem:non-compact-dimension-deduce} \ref{item:non-compact-dimension-deduce-exist-solvable}.
          Then $H$ and $G / H$ are connected solvable Lie groups, we have $I ( H )$ and $I ( G / H )$ by the induction hypothesis.
          Thus, we also have $I ( G )$ by \ref{item:non-compact-dimension-bound-exact}.

    \item
          We prove it by induction on $\dim G$.
          If $\dim G = 0$, then one has $r ( G ) = 0$ and hence
          \begin{align}
            d ( G ) \geq 0 = d ( \mathbb{R} ) ( \dim G - r ( G ) ).
          \end{align}
          If $\dim G = 1$, then $G$ is a solvable Lie group and hence $I ( G )$ follows from \ref{item:non-compact-dimension-bound-solvable}.

          Now we show $I ( G )$ when $\dim G \geq 2$.
          If $G$ satisfies the condition \ref{item:non-compact-dimension-deduce-exist-solvable-condition}, then $I ( G )$ follows from \ref{item:non-compact-dimension-bound-exact}.
          Thus, it suffices to show $I ( G )$ when $G$ does not satisfy the condition \ref{item:non-compact-dimension-deduce-exist-solvable-condition}.
          By Lemma \ref{lem:non-compact-dimension-deduce} \ref{item:non-compact-dimension-deduce-exist}, there exists a closed subgroup $H \in \mathcal{A}$ of $G$ such that $\dim H < \dim G$ and \eqref{eq:non-compact-dimension-deduce-exist-state}.
          We have $I ( H )$ by the induction hypothesis and hence
          \begin{align}
            d ( G )
            \geq d ( H )
            \geq d ( \mathbb{R} ) ( \dim H - r ( H ) )
          \end{align}
          by \eqref{eq:non-compact-dimension-bound-order-assume} and \eqref{eq:non-compact-dimension-bound-order-simple}.
          Thus, we obtain $I ( G )$ by \eqref{eq:non-compact-dimension-deduce-exist-state}.
          \qedhere

  \end{enumerate}

\end{proof}

\begin{proof}[Proof of Corollary \ref{cor:Young-non-compact-dimension}]

  Let $d ( G ) := - \ln ( Y ( p_1 , p_2 ; G ) )$ for locally compact group $G$.
  Then
  \begin{align}
    d ( G )
     & := - \ln ( Y ( p_1 , p_2 ; G ) )                                  \\
     & \geq - \ln ( Y ( p_1 , p_2 ; H ) Y ( p_1 , p_2 ; G / H ) )        \\
     & = - \ln ( Y ( p_1 , p_2 ; H ) ) - \ln ( Y ( p_1 , p_2 ; G / H ) ) \\
     & = d ( H ) + d ( G / H )
  \end{align}
  holds for any connected closed normal subgroup $H \lhd G$ and hence we have \eqref{eq:non-compact-dimension-bound-normal}.
  In addition,
  \begin{align}
    d ( G )
    := - \ln ( Y ( p_1 , p_2 ; G ) )
    \geq - \ln ( Y ( p_1 , p_2 ; H ) )
    = d ( H ) \label{eq:d-order}
  \end{align}
  holds for any closed subgroup $H \subset G$ by Theorem \ref{thm:Young-order-reversing}.
  We have
  \begin{align}
    d ( H )
    := - \ln ( Y ( p_1 , p_2 ; H ) )
    \geq 0 \label{eq:d-positive}
  \end{align}
  by Example \ref{ex:Young-order-reversing-example} \ref{item:Young-order-reversing-example-classic}, and the equality holds for $H := \mathbb{R} / \mathbb{Z}$ by Corollary \ref{cor:Young-R-compare}.
  Since \eqref{eq:non-compact-dimension-bound-order-assume} follows from \eqref{eq:d-order} and \eqref{eq:d-positive}, we have \eqref{eq:non-compact-dimension-bound-order-simple} for $G \in \mathcal{A}$ by Lemma \ref{lem:non-compact-dimension-bound} \ref{item:non-compact-dimension-bound-order}.
  Thus, the inequality
  \begin{align}
    Y ( p_1 , p_2 ; G )
    = e^{- d ( G )}
    \leq e^{- d ( \mathbb{R} ) ( \dim G - r ( G ) )}
    = Y ( p_1 , p_2 ; \mathbb{R} )^{\dim G - r ( G )}
  \end{align}
  is obtained.
\end{proof}

\section*{Acknowledgement}

This work was supported by JSPS KAKENHI Grant Number JP19J22628 and Leading Graduate Course for Frontiers of Mathematical Sciences and Physics (FMSP).
The author would like to thank his advisor Toshiyuki Kobayashi for his support.
The author is also grateful to Yuichiro Tanaka, Toshihisa Kubo, and the anonymous referees for their careful comments.

\printbibliography

\noindent
Takashi Satomi: Graduate School of Mathematical Sciences, The University of Tokyo, 3-8-1 Komaba Meguro-ku Tokyo 153-8914, Japan.

\noindent
E-mail: tsatomi@ms.u-tokyo.ac.jp

\end{document}